\documentclass[11pt]{amsart}
\usepackage{amssymb}
\usepackage{amsmath}
\usepackage{amsthm}
\usepackage{amscd}
\usepackage{savesym}
\usepackage{xypic}
\savesymbol{cir}
\usepackage{amsrefs}
\restoresymbol{Ref}{cir}

  1

\newcommand{\bq}{\mathbb Q}

\newcommand{\bz}{\mathbb Z}

\newcommand{\bF}{\mathbb F}

\newcommand{\co}{\mathcal O}

\newcommand{\et}{\mathrm{et}}

\DeclareMathOperator{\Spec}{Spec} 
 \DeclareMathOperator{\trace}{Tr}

\DeclareMathOperator{\speci}{sp}

 \DeclareMathOperator{\Gal}{Gal}

\begin{document}

\newtheorem{theorem}{Theorem}[section]
\newtheorem{lemma}{Lemma}[section]
\newtheorem{prop}{Proposition}[section]
\newtheorem{cor}{Corollary}[section]
\newtheorem{defi}{Definition}[section]
\newtheorem{rem}{Remark}[section]
\newtheorem{exam}{Example}[section]
\newtheorem{conjecture}{Conjecture}[section]

\numberwithin{equation}{section}

\title[On the $p$-adic local invariant cycle theorem]{On the p-adic local invariant cycle theorem}

\author [Yi-Tao Wu]{Yi-Tao Wu}
\thanks{ 2010 Mathematics Subject Classification 14F20 14F30 \\ Supported by DFG SFB1085(Higher Invariants)}
\maketitle \vspace{-0.6cm}

\bigskip

\begin{abstract} For a proper, flat, generically smooth scheme $X$ over a complete DVR with finite residue field of characteristic $p$, we define a specialization morphism from the rigid cohomology of the geometric special fibre to $D_{crys}$ of the $p$-adic \'etale cohomology of the geometric generic fibre, and we make a conjecture ("$p$-adic local invariant cycle theorem") that describes the behavior of this map for regular $X$, analogous to the situation in $l$-adic \'etale cohomology for $l\neq p$. Our main result is that, if $X$ has semistable reduction, this specialization map induces an isomorphism on the slope $[0,1)$-part.  \end{abstract}

\bigskip

\section{\bf Introduction}

\subsection{} Let $R$ be a complete discrete valuation
ring with quotient field $K$ and finite residue field $k$ of
characteristic $p$. We choose a separable closure $\bar{K}$ of $K$ (resp. $\bar{k}$ of $k$)
and set $S=\Spec(R)$, $\eta=\Spec(K)$, $s=\Spec(k)$, $\bar{\eta}=\Spec(\bar{K})$ and $\bar{s}=\Spec(\bar{k})$.
We denote by $I\subseteq G :=
\Gal(\bar{K}/K)$ the inertia subgroup, by $W(k)$ the ring of Witt
vectors of $k$, and by $K_0$ (resp. $\hat{K}_0^{ur}$) the fraction field of $W(k)$ (resp. $W(\bar{k})$).

For any scheme $X$ over $S$, we have a \textsl{ geometric special fiber} $X_{\bar{s}}$
over $\bar{k}$ and a \textsl{geometric generic fiber} $X_{\bar{\eta}}$ over $\bar{K}$, i.e.
a commutative diagram
\[
\begin{CD}
X_{\bar{s}} @>>> X @<<< X_{\bar{\eta}}\\
@VVV   @VVV    @VVV\\
\bar{s}  @>>> S @<<< \bar{\eta}
\end{CD}
\]
where the squares are Cartesian.  If $X\to S$ is moreover proper and flat one has the
specialization morphism on $l$-adic \'etale cohomology groups:
$$ \speci: H^i(X_{\bar{s}}, \mathbb{Q}_l) \to H^i(X_{\bar{\eta}},\mathbb{Q}_l)^I$$
which is the composition
$$ H^i(X_{\bar{s}}, \mathbb{Q}_l) \cong H^i(X', \mathbb{Q}_l) \to H^i(X'_{\eta}, \mathbb{Q}_l) \to H^i(X_{\bar{\eta}},\mathbb{Q}_l)^I$$
where $X'$ is the base change of $X$ to a strict Henselization of
$S$ at $\bar{s}$ and the first isomorphism results from proper base change.
This map $\speci$ is $G$-equivariant. We note that all Frobenius eigenvalues on both source and target are known to be Weil numbers,
so we have a natural increasing weight filtration $W_j$ on both sides.

\begin{conjecture} ("local invariant cycle theorem") Assume $X$ is regular and $l\neq p$. Then
\begin{enumerate}
\item $\speci$ is an epimorphism
\item $\speci$ induces an isomorphism on $W_{i-1}$
\end{enumerate}
\label{ladicinvcyc}\end{conjecture}

Since by \cite{weilii} one has $W_iH^i(X_{\bar{s}}, \mathbb{Q}_l)=H^i(X_{\bar{s}},
\mathbb{Q}_l)$, Conjecture \ref{ladicinvcyc}
implies that the kernel of $\speci$ is pure of weight $i$. In
\cite{weilii}(3.6), Deligne has proved the local invariant cycle
theorem in the equal characteristic case $\text{char}(K)=p$.  In the
mixed characteristic case $\text{char}(K)=0$, it is well known that
the local invariant cycle theorem is implied by Deligne's
conjecture on the purity of the monodromy filtration (the "monodromy
weight conjecture"). We refer to \cite{illusie} for this implication
in the semistable case and to \cite{flachmorin} in the general
regular case. Thus the local invariant cycle theorem also holds in
the mixed characteristic case if $\dim(X_\eta)\leq 2$ by the results
of Rapoport and Zink \cite{rapozink82} and in many more cases by the
recent work of Scholze \cite{scholze11}. Further unconditional, and
probably well known, results were recorded by Flach and Morin in
\cite{flachmorin}: If $X$ is regular and $l\neq p$ then $\speci$
induces an isomorphism on $W_1$ and is an isomorphism for $i=0,1$.


\subsection {} From now on and for the rest of this paper we assume that $K$ has characteristic $0$ and that $X\to S$ is proper and flat.  For $l=p$, it is then natural
to ask about a local invariant cycle theorem for $p-$adic
cohomology. In the case where $X\to S$ is {\em smooth} it was shown in
\cite{fm87}(4.1) that the map $\speci:
H^i(X_{\bar{s}}, \mathbb{Q}_p) \to H^i(X_{\bar{\eta}},
\mathbb{Q}_p)^I$ is an isomorphism
as a consequence of Fontaine's $C_{crys}$-comparison isomorphism in
$p$-adic Hodge theory proven by Fontaine-Messing \cite{fm87} and
Faltings \cite{fal89}.

However, it is well known that the geometric $p$-adic \'etale cohomology of
varieties over $k$ only describes the slope $0$ part of the full
$p$-adic (Weil) cohomology, which is Berthelot's rigid cohomology
$H^i_{rig}(X_s/k)$ . To have a context which is fully analogous to
the $l$-adic situation we need to construct an enlarged specialization
map as following. Here and in the following we refer to \cite{fon1},
\cite{fon2} for Fontaine's functors $D_{crys}$, $D_{st}$, $D_{pst}$,
$D_{dR}$ and the corresponding period rings.

\begin{prop}
If $X\to S$ is proper, flat and generically smooth, then there
is a functorial $\phi$-equivariant map
\begin{equation}\speci' : H^i_{rig}(X_s/k) \to D_{crys}(H^i(X_{\bar{\eta}}, \mathbb{Q}_p))\label{speciprime}\end{equation}
and a commutative diagram of $Gal(\bar{k}/k)-$modules,
we also have isomophisms
$$\lambda_s:
H^i(X_{\bar{s}},\mathbb{Q}_p) \cong
(H^i_{rig}(X_s/k)\otimes_{K_0}\hat{K_0}^{ur})^{\phi\otimes\phi=1} =:
H^i_{rig}(X_s/k)^{slope 0} $$ and
$$\lambda_{\eta}: H^i(X_{\bar{\eta}},\mathbb{Q}_p)^I \cong
(D_{crys}(H^i(X_{\bar{\eta}},\mathbb{Q}_p))\otimes_{K_0}\hat{K_0}^{ur})^{\phi\otimes\phi=1}
= :D_{crys}(H^i(X_{\bar{\eta}},\mathbb{Q}_p))^{slope 0}.$$
\label{speciconstruct}\end{prop}

In the case where $X$ has semistable reduction the map $\speci'$ is the composite of a map
\[ H^i_{rig}(X_s/k)\to H^i_{HK}(X/S) \]
constructed by Chiarelotto in \cite{chia99}, and the $N=0$-part of
Fontaine's $C_{st}$-comparison isomorphism
\[H^i_{HK}(X/S)\cong D_{st}(H^i(X_{\bar{\eta}},\mathbb{Q}_p))\]
proved by Tsuji \cite{tsuji99}, among others. Here $H^i_{HK}(X/S)$
is the log-crystalline cohomology defined by Hyodo and Kato
\cite{hyka94}. In general, we construct $\speci'$ by descent from
the semistable case using de Jong's alterations \cite{dj96}. We
believe that the $p$-adic Hodge theory of rigid analytic spaces that is currently being developed by various authors \cite{scholze12}, \cite{Bellovin} should ultimately give a more sheaf theoretic construction of
$\speci'$ which then also works if $f$ is only proper and flat.
We note that a map
\[ H^i_{rig}(X_s/k) \to H^i_{dR}(X_\eta/K)\cong D_{dR}(H^i(X_{\bar{\eta}}, \mathbb{Q}_p))\]
is more or less immediate from the definition of rigid cohomology as
de Rham cohomology of a tube (see \cite{bcf04}[Thm. 6.6] for the general proper, flat case where this map is called a cospecialization map).

Now note that the eigenvalues of the $k$-linear Frobenius
$\phi^{[k:\mathbb{F}_p]}$ on $H^i_{rig}(X_s/k)$ are Weil numbers,
and following the argument of \cite{flachmorin}(section 10) in the
$l$-adic case one proves that the same is true for
$D_{pst}(H^i(X_{\bar{\eta}}, \mathbb{Q}_p))$. Hence one deduces
$\phi$-stable weight filtrations $W_j$ on both $H^i_{rig}(X_s/k)$
and
$$D_{crys}(H^i(X_{\bar{\eta}}, \mathbb{Q}_p))=
D_{st}(H^i(X_{\bar{\eta}},
\mathbb{Q}_p))^{N=0}=D_{pst}(H^i(X_{\bar{\eta}}, \mathbb{Q}_p))^{G,
N=0}.$$ The full $p$-adic analogue of Conjecture \ref{ladicinvcyc} (and the results mentioned after it) is then the following conjecture.

\begin{conjecture} ("p-adic local invariant cycle theorem") If $X$ is regular then
\begin{enumerate}
\item $\speci'$ is an epimorphism
\item $\speci'$ induces an isomorphism on $W_{i-1}$
\item $\speci'$ is an isomorphism for $i=0,1$
\end{enumerate}
\label{padicinvcyc}\end{conjecture}

For semistable $X$ part (1) and (3) were also conjectured by Chiarellotto \cite{chia99}[\S 4] and proven by him by assuming the $p$-adic monodromy weight conjecture stated in \cite{mok93}[Conj. 3.27]. This conjecture is known if $\dim(X_\eta)\leq 2$ and for $i=0,1$ in general. So part (3) holds for semistable $X$.
It can also be seen from his paper that $\speci'$ is an isomorphism for $i=0,1$ in the semistable reduction case.

For general regular $X$ it seems difficult to prove Conjecture \ref{padicinvcyc} without a different construction of the map $\speci'$ that embeds it into a suitable long exact  Clemens-Schmid sequence. It would also be interesting to say something more about the slopes of eigenvalues that occur in the kernel of $\speci'$ for regular $X$. Conjecture \ref{padicinvcyc} says that such eigenvalues are of weight $i$.

The main result of this article is then the following Theorem.
\begin{theorem} If $X$ has semistable reduction, then the map $\speci'$ in (\ref{speciprime}) induces an isomorphism on the slope $[0,1)$-part, i.e. $  H^i(X_s, W\co_{X_s})_\bq \simeq D_{crys}(H^i(X_{\bar{\eta}},\mathbb{Q}_p))^{[0,1)} .$
\label{slope01semistable}\end{theorem}

The following conjecture is the generalization of Theorem \ref{slope01semistable}

\begin{conjecture} If $X$ is regular, the map $\speci'$ in (\ref{speciprime}) induces an isomorphism on the slope $[0,1)$-part.
	\label{slope01}\end{conjecture}	
	
    Note that this conjecture implies that: if $X$ is regular then the map $\speci': H^i(X_{\bar{s}}, \mathbb{Q}_p) \to H^i(X_{\bar{\eta}},
\mathbb{Q}_p)^I$ is an isomorphism.

    Conjecture \ref{slope01} is equivalent to a "compatibility of trace maps", i.e. the commutativity of the following diagram
    \begin{equation}
    \begin{CD}
    H^i(X^{(1)}_s, W\co_{X^{(1)}_s})_{\mathbb{Q}} @>sp'>>
    D_{crys}(H^i(X^{(1)}_{\bar{\eta}},\mathbb{Q}_p))^{[0,1)} \\
    @V\tau_{s}^{[0,1)}VV   @V\tau_{\eta}VV \\
    H^i(X_s, W\co_{X_s})_{\mathbb{Q}} @>sp'>>
    D_{crys}(H^i(X_{\bar{\eta}},\mathbb{Q}_p))^{[0,1)}
    \end{CD}
    \label{trace compatibility slope01}
    \end{equation}

    where $\tau_{s}^{[0,1)}$ is the trace map defined by Berthelot, Esnault and R\"ulling in
    \cite{beresrue10}, and $\tau_{\eta}$ is the trace map induced from
    the lemma \ref{tracelemma}.

    This paper is organized as follows. We first give some preliminaries on $\phi$-modules and construct $\lambda_s$, $\lambda_\eta$ in section \ref{preliminaries}, then we prove
 prove Proposition \ref{speciconstruct} and Theorem \ref{slope01} in section \ref{specisect}.

 {\bf Acknowledgement} It's my great pleasure to thank my advisor Matthias Flach for advising the author's PhD thesis(which was the genesis of this manuscript) and for the clarifications, corrections, discussions and constant encouragement.






\section{Preliminaries on $\phi$-modules and construction of $\lambda_s$, $\lambda_\eta$}\label{preliminaries}

As in the introduction we let $\phi$ be the absolute Frobenius on either $K_0$ or $\hat{K}_0^{ur}$, and a {\em $\phi$-module} over either field is a finite dimensional vector space with a bijective $\phi$-semilinear endomorphism $\phi$. The Dieudonne-Manin classification \cite{manin63} tells us that the category of $\phi$-modules over $\hat{K}_0^{ur}$ is semisimple and that the simple objects $$E_q=E_{r,s}:=\hat{K}_0^{ur}[\phi]/(\phi^r-p^s)$$ are parametrized by the set of rational numbers $q=\frac{s}{r} \in \mathbb{Q}$, called slopes in this context. For $q\in\bq$ and a $\phi$-module $D$ over $K_0$, the $q$-isotypical part of
$D\otimes_{K_0}\hat{K}_0^{ur}$ descends to $K_0$ and is called the slope $q$-part of $D$. Denote this $\phi$-submodule of $D$ by $D^{[q]}$. It is canonically a direct summand complemented by the sum of $D^{[q']}$ for $q'\neq q$. If ${\mathcal I}\subseteq\bq$ is any interval (open, closed or half-open) one defines $D^{\mathcal I}$ to be the sum of $D^{[q]}$ for $q\in \mathcal I$.

For any $\phi$-module D $$V(D):=(D\otimes_{K_0}\hat{K_0}^{ur})^{\phi\otimes\phi=1}$$ is a finite-dimensional $\bq_p$-vector space with a continuous $\Gal(\bar{k}/k)$-action, and the functor $V$ gives a Fontaine-style equivalence of categories between such representations and $\phi$-modules $D$ over $K_0$ of slope $0$, i.e. such that $D=D^{[0]}$. For any $D$ one has $V(D)=V(D^{[0]})$
and for us it is more convenient to define the slope $0$-part by this formula, i.e. we set
$$ D^{slope 0}:=V(D).$$
Of course the $\Gal(\bar{k}/k)$-action on $V(D)$ again just amounts to a single automorphism $\text{Frob}_k=1\otimes\phi^{[k:\bF_p]}=\phi^{-[k:\bF_p]}\otimes 1$.
\begin{lemma}
Let V be a finite dimensional $\mathbb{Q}_p$-vector space with a
continuous $G:=\Gal(\bar{K}/K)$-action, and such that
$D_{dR}(V)/Fil^0D_{dR}(V)=0$. Then we have an isomorphism
$$V^I \simeq D_{crys}(V)^{slope 0}.$$
\end{lemma}

\begin{proof} Let $L$ be the finite extension of $\hat{K}^{ur}_0$ fixed by $I$. Then
\begin{equation}D_{dR,L}(V)=(B_{dR}\otimes_{\bq_p}V)^I\cong D_{dR}(V)\otimes_KL=Fil^0D_{dR}(V)\otimes_KL=(B^0_{dR}\otimes_{\bq_p}V)^I\label{e1}\end{equation} by assumption. Moreover
$(B_{cris}\otimes_{\bq_p}V)^I\cong D_{crys}(V)\otimes_{K_0}\hat{K}^{ur}_0$
and
$$ (B^{\phi=1}_{cris}\otimes_{\bq_p}V)^I\cong (D_{crys}(V)\otimes_{K_0}\hat{K}^{ur}_0)^{\phi\otimes\phi=1}=D_{crys}(V)^{slope 0}.$$
By (\ref{e1}) the image of $(B^{\phi=1}_{cris}\otimes_{\bq_p}V)^I$ in $(B_{dR}\otimes_{\bq_p}V)^I$ lies in $(B^0_{dR}\otimes_{\bq_p}V)^I$, i.e. in $(Fil^0B_{crys}^{\phi=1}\otimes_{\bq_p}V )^{I}$. But by \cite{fon1}[Thm. 5.3.7]
we have $Fil^0B_{crys}^{\phi=1}=\bq_p$, hence we obtain $V^I\simeq D_{crys}(V)^{slope 0}$.
\end{proof}

\begin{cor} For any variety $X_\eta$ over $\Spec(K)$ one has a functorial isomorphism
$$\lambda_\eta:H^i(X_{\bar{\eta}}, \mathbb{Q}_p)^I \simeq
D_{crys}(H^i(X_{\bar{\eta}},\mathbb{Q}_p))^{slope 0}$$
\end{cor}

\begin{proof} By the $C_{dR}$-isomorphism (see for example \cite{beil13} for the case of arbitrary $X_\eta$) we have $D_{dR}(H^i(X_{\bar{\eta}}, \mathbb{Q}_p))\cong
H^i_{dR}(X_\eta/K)$ and one always has $Fil^0H^i_{dR}(X_\eta/K) = H^i_{dR}(X_\eta/K)$.
\end{proof}

For any proper variety $X_s$ over $\Spec(k)$ Berthelot, Bloch and Esnault \cite{bbe07} have defined a functorial morphism of $\phi$-modules
\begin{equation} H^i_{rig}(X_s/k)\to H^i(X_s, W\co_{X_s})_\bq \label{bbe}\end{equation}
where $H^i(X_s, W\co_{X_s})_\bq=\varprojlim_n H^i(X_s, W_n\co_{X_s})\otimes_\bz\bq$ is Witt vector cohomology, and they show \cite{bbe07}[Thm. 1.1] that this morphism induces an isomorphism
\begin{equation}H^i_{rig}(X_s/k)^{[0,1)}\simeq H^i(X_s, W\co_{X_s})_\bq.\label{bbeiso}\end{equation}

\begin{lemma} For a proper variety $X_s$ over $\Spec(k)$ one has a functorial isomorphism
$$\lambda_s:H^i(X_{\bar{s}}, \mathbb{Q}_p) \simeq
H^i_{rig}(X_s/k)^{slope 0}.$$
\label{bbelemma}\end{lemma}

\begin{proof} On $X_{\bar{s},\et}$ one has an Artin-Schreier type short exact sequence \cite{illusie79}[Prop.3.28]
$$0 \longrightarrow \mathbb{Z}_p \longrightarrow W\co_{X_{\bar{s}}} \stackrel{1-\phi}{\longrightarrow} W\co_{X_{\bar{s}}} \longrightarrow 0 $$
whose induced long exact cohomology sequence splits into short exact sequences after tensoring with $\bq$ by \cite{illusie79}[Lemma 5.3]. Hence using the scalar extension of (\ref{bbeiso}) to $\hat{K}_0^{ur}$ we get
\begin{align*}H^i(X_{\bar{s}}, \mathbb{Q}_p) &=
H^i(X_{\bar{s}},W\co_{X_{\bar{s}}})^{\phi=1}_\bq \simeq
(H^i_{rig}(X_{\bar{s}}/\bar{k})^{[0,1)})^{\phi=1} \\
&= (H^i_{rig}(X_{\bar{s}}/\bar{k}))^{\phi=1} \simeq
(H^i_{rig}(X_s/k)\otimes_{K_0}\hat{K_0}^{ur})^{\phi\otimes\phi=1}=H^i_{rig}(X_s/k)^{slope 0}.\end{align*}
\end{proof}

Combining the above lemmas, we get the isomorphisms $\lambda_s$ and $\lambda_\eta$ in Proposition \ref{speciconstruct}.

\section{\bf Construction of the $p$-adic specialization map}\label{specisect}

It remains to construct $\speci'$. We do this first for semistable $X$ and then use de Jong's alteration to descend to the general case.

\subsection{The semistable case} In this section we assume that $X$ is semistable, i.e. its special fibre
$$X_s =: Y = \cup_{\iota\in I} Y_\iota $$ is a reduced, normal crossing divisor with smooth proper irreducible components $Y_\iota$. As usual, we let $$Y^{(j)}=\coprod_{\{\iota_1,\dots,\iota_j\}\subseteq I}Y_{\iota_1} \cap \cdots \cap Y_{\iota_j} $$
be the disjoint union of the intersections of $j$ components, which is a smooth and proper $k$-scheme.

\subsubsection{Hyodo-Kato cohomology} The Hyodo-Kato cohomology of $X/S$ (or log-crystalline cohomology of the log smooth morphism $X_s\to\Spec(k)$ where these schemes are endowed with the log-structure induced by the log structures on $X$ and $S$ given by the special fibre) can be computed by the Hyodo-deRham-Witt complex $W\omega^{\bullet}_{Y}$
\begin{equation} H^i_{HK}(X/S)\cong H^i(Y_\et,W\omega^{\bullet}_{Y})_\bq \label{hkiso}\end{equation}
on $Y_\et$ which is described in detail in \cite{hyka94}[(1.1), Thm 4.19]. The complex $$W\omega^{\bullet}_{Y}=\varprojlim_n W_n\omega^{\bullet}_{Y}$$
is a pro-complex where each $W_n\omega^{q}_{Y}$ is a coherent $W_n\co_{Y_{et}}$-module and by definition one has
\begin{equation}W_n\omega^0_{Y}=W_n\co_{Y_{et}}.\label{deg0}\end{equation}
The morphism of pro-complexes $W\omega^{\bullet}_{Y}\to W\co_{Y_{et}}[0]$ then induces a map analogous to (\ref{bbe})
\begin{equation} H^i_{HK}(X/S)\to H^i(Y_\et, W\co_{Y_\et})_\bq\cong H^i(X_s, W\co_{X_s})_\bq \label{lorenzonmap}\end{equation}
and one has the following analogue of (\ref{bbeiso}) and of Lemma \ref{bbelemma}.

\begin{lemma} For a semistable scheme $X/S$ the map (\ref{lorenzonmap}) induces a functorial isomorphism
\begin{equation}H^i_{HK}(X/S)^{[0,1)}\simeq H^i(X_s, W\co_{X_s})_\bq\label{lorenzoniso}\end{equation}
and a functorial isomorphism
$$\tilde{\lambda}_s : H^i(X_{\bar{s}}, \mathbb{Q}_p) \cong H_{HK}^i(X/S)^{slope 0}. $$
Moreover, we have $H^i_{HK}(X/S)^{N=0,[0,1)}\simeq H^i_{HK}(X/S)^{[0,1)}$.
\label{lorenzonlemma}\end{lemma}

\begin{proof} The first statement follows from the degeneration at $E_1$ of the slopes spectral sequence
\[ E_1^{qr}=H^r(Y_\et,W\omega_Y^q)_\bq\Rightarrow H^{r+q}(Y_\et, W\omega_Y^\bullet)_\bq\]
proven in \cite{lor02}[Thm. 3.1] (see the remark in loc. cit. after (3.1.1)). The second statement follows from the first by exactly the same proof as that of Lemma \ref{bbelemma}. The third statement is an easy consequence of the relation $N\phi=p\phi N$ and the fact that $H^i_{HK}(X/S)$ has no negative slopes.
 \end{proof}

\subsubsection{The Hyodo-Steenbrink complex} In this section we relate rigid and Hyodo-Kato cohomology of $X_s$ following \cite{chia99}. By \cite{mok93}[Cor. 3.17] one has a quasi-isomorphism
\begin{equation}W_n\omega^{\bullet}_{Y}\cong W_nA^{\bullet}\label{waiso}\end{equation} where $W_nA^{\bullet}$ is the simple complex associated to the (Hyodo-Steenbrink) bicomplex
$W_nA^{i,j}$, $i,j \geq 0$, of sheaves in $Y_{et}$ given by
$$W_nA^{i,j} = \frac{W_n\tilde{\omega}^{i+j+1}_Y}{P_jW_n\tilde{\omega}^{i+j+1}_Y}.$$
Here $ W_n\tilde{\omega}^{\bullet}_Y$ is a complex defined in \cite{hyka94}[(1.4)] which carries a weight filtration $P_jW_n\tilde{\omega}^{\bullet}_Y$ defined in \cite{mok93}[Sec. 3.5]. Moreover there is a certain global section $\theta$ of
$W_n\tilde{\omega}^{1}_Y$ and the isomorphism (\ref{waiso}) is given by multiplication with $\theta$ \cite{mok93}[Prop. 3.15]. For example, using (\ref{deg0}), the projection to the degree $0$ part of both complexes in (\ref{waiso}) gives rise to a commutative diagram
\begin{equation}\begin{CD} W_n\omega^{\bullet}_{Y}@>\sim>> W_nA^{\bullet}\\@VVV @VVV\\
W_n\co_{Y_{et}} @>\wedge\theta >> W_n\tilde{\omega}^{1}_Y/P_0W_n\tilde{\omega}^{1}_Y\end{CD}\label{0proj}\end{equation}

Let $\nu_n$ be the endomorphism induced on the simple complex by
the endomorphism on $W_nA^{\bullet\bullet}$ given by the natural projection
$W_nA^{i,j} \to W_nA^{i-1,j+1}$ multiplied with $(-1)^{i+j+1}$. By \cite{mok93}[Prop. 3.18] the inverse limit
$$\nu = \varprojlim_n \nu_n $$
induces the monodromy operator $N$ in Hyodo-Kato cohomology via (\ref{waiso}) and (\ref{hkiso}).
It is clear that the kernel of $\nu_n$ on $W_nA^\bullet$ is the simple complex associated to the double subcomplex
$$\frac{P_{j+1}W_n\tilde{\omega}^{i+j+1}_Y}{P_jW_n\tilde{\omega}^{i+j+1}_Y}\subseteq W_nA^{i,j}.$$
On the other hand by \cite{mok93}(3.7) one has an isomorphism of complexes
$$Res : Gr_j^{W} W_n\tilde{\omega}^{\bullet}_Y[j] \to W_n\Omega^{\bullet}_{Y^{(j)}}(-j) $$
where $W_n\Omega^{\bullet}_{Y^{(j)}}$ is the usual de Rham-Witt complex of $Y^{(j)}$, thought of as a complex on $Y_{et}$ via the natural finite morphism $Y^{(j)}\to Y$, $(-j)$ is the Tate-shift related to the Frobenius
structure \cite{illusie79} and $Gr_j=P_j/P_{j-1}$. This leads to the following identification of the kernel of $\nu$ due to Chiarellotto.

\begin{prop}(\cite{chia99},Prop.1.8) The kernel of the operator
$$\nu_n : W_nA^{\bullet} \to  W_nA^{\bullet} $$
is isomorphic to the simple complex associated to the double complex $W_n\Omega^{\bullet}_{Y^{(\bullet)}}$
$$0 \to W_n\Omega^{\bullet}_{Y^{(1)}} \stackrel{\rho_1}{\to} W_n\Omega^{\bullet}_{Y^{(2)}} \stackrel{\rho_2}{\to} W_n\Omega^{\bullet}_{Y^{(3)}} \cdots $$
on $Y_{et}$, where $\rho_j : W_n\Omega^{\bullet}_{Y^{(j)}} \to
W_n\Omega^{\bullet}_{Y^{(j+1)}}$ is defined by
$$\rho_j = (-1)^j\sum_{1\leq i \leq j+1} (-1)^{i+1}\delta_i^* $$
and $\delta_i: Y^{(j+1)} \to Y^{(j)}$ is the inclusion $$Y_{\iota_1} \cap \cdots \cap
Y_{\iota_j} \hookrightarrow Y_{\iota_1} \cap \cdots \cap Y_{\iota_{i-1}} \cap
Y_{\iota_{i+1}} \cap \cdots \cap Y_{\iota_j}.$$
\end{prop}

For the proper, smooth $k$-schemes $Y^{(j)}$ we have
\begin{equation} H^i_{rig}(Y^{(j)}/k)\cong H^i_{crys}(Y^{(j)}/k) \cong H^i(Y_\et,W\Omega^{\bullet}_{Y^{(j)}})_\bq\label{propsmooth}\end{equation}
and the simplicial scheme $Y^{(j)}$ with boundary maps $\delta_i$ is a proper, smooth hypercovering of $Y$. Since rigid cohomology satisfies cohomological descent \cite{tsuzuki03} we obtain an isomorphism
\[ H^i_{rig}(Y/k)\cong H^i(Y_\et,W\Omega^{\bullet}_{Y^{(\bullet)}})_\bq\cong H^i(Y_\et,\text{ker}(\nu)^\bullet)_\bq\]
and the exact sequence of complexes
$$0 \to \text{ker}(\nu)^\bullet \to WA^{\bullet} \stackrel{\nu}{\to} WA^{\bullet} $$
then induces Chiarellotto's map
\begin{equation} H^i_{rig}(Y/k)\to H^i(Y_\et,WA^\bullet)^{\nu=0}_\bq\cong H^i_{HK}(X/S)^{N=0}.\label{chiamap}\end{equation}
Using Fontaine-Jannsen's $C_{st}$-comparison isomorphism proven by Tsuji \cite{tsuji99}
$$H^i_{HK}(X/S)\otimes_{K_0}B_{st}\cong H^i(X_{\bar{\eta}}, \mathbb{Q}_p) \otimes_{\bq_p} B_{st} $$
we obtain an isomorphism
\[c_{st}:H^i_{HK}(X/S)\cong D_{st}(H^i(X_{\bar{\eta}},\bq_p))\]
and hence our $p$-adic specialization map
\[\speci' : H^i_{rig}(X_s/k) \to D_{st}(H^i(X_{\bar{\eta}},\bq_p))^{N=0}=D_{crys}(H^i(X_{\bar{\eta}},\bq_p)).\]
Composing (\ref{chiamap}) with (\ref{lorenzonmap}) we obtain a map
\begin{equation} H^i_{rig}(Y/k)\to H^i_{HK}(X/S)^{N=0}\to H^i(X_s, W\co_{X_s})_\bq.\label{bbe2}\end{equation}

\begin{lemma} The map (\ref{bbe2}) coincides with the map (\ref{bbe}) defined by Berthelot, Bloch and Esnault in \cite{bbe07}. In particular, we have a commutative diagram
\[\begin{CD}
H^i(X_{\bar{s}}, \mathbb{Q}_p) @= H^i(X_{\bar{s}}, \mathbb{Q}_p)\\
@V\lambda_sV\simeq V   @V\tilde{\lambda}_sV\simeq V\\
H^i_{rig}(X_s/k)^{slope 0} @>(\ref{chiamap})>>H^i_{HK}(X/S)^{N=0,slope 0}.
\end{CD}\]
\end{lemma}

\begin{proof} One has maps of spectral sequences
\[\begin{CD} E_1^{q,p}=H^p_{rig}(Y^{(q)}/k) @.\Rightarrow @. H^{p+q}_{rig}(Y/k)\\ @VVV @. @VV(\ref{chiamap}) V\\
E_1^{q,p}=H^p(Y_\et, WA^{\bullet,q})_\bq @.\Rightarrow @. H^{p+q}_{HK}(X/S)\\ @VVV @. @VV(\ref{lorenzonmap}) V\\
E_1^{q,p}=H^p(Y_\et,W\co_{Y^{(q)}})_\bq @.\Rightarrow@.  H^{p+q}(Y_\et,W\co_Y)_\bq.
\end{CD}\]
The first vertical map is induced by the inclusion of double complexes
\[W\Omega^{\bullet}_{Y^{(\bullet+1)}}\cong \text{ker}(\nu_n)^{\bullet,\bullet}\subseteq W_nA^{\bullet,\bullet}\]
and filtering both complexes "vertically", and using (\ref{propsmooth}). The second vertical map is obtained by
projecting complexes onto the degree $0$ term. Now the map (\ref{bbe}) of Berthelot, Bloch and Esnault is functorial and is induced by projecting the de Rham-Witt complex onto its degree $0$ term if the scheme is proper and smooth. Hence the composite map of spectral sequences equals (\ref{bbe}) on the initial term and therefore also on the end term.
\end{proof}

By these we get Theorem \ref{slope01semistable}.

\subsection{General case} We now define the map $\speci':H^i_{rig}(X_s/k) \to D_{crys}(H^i(X_{\bar{\eta}}, \mathbb{Q}_p))$ in the case where $X$ is proper, flat and generically smooth over $S$. Recall from \cite{dj96} that an {\em alteration} $g:Y'\to Y$ is a proper, surjective morphism of schemes for which there exists a dense Zariski open $U\subseteq Y$ so that $g^{-1}(U)\xrightarrow{g} U$ is finite \'etale. One has the following fundamental theorem \cite{dj96}[6.5]

\begin{theorem} If $Y$ is a proper, flat scheme over $S$ then there exists an alteration $Y'\to Y$ so that $Y'$ is a semistable family (over the possibly non-connected regular base scheme $S'=\Spec(\Gamma(Y',\co_{Y'}))$).
\end{theorem}

We call such a morphism $Y'\to Y$ a {\em semistable alteration}. Using this theorem we can construct a commutative diagram
\begin{equation}
\begin{CD}
X_{s^{(2)}}^{(2)} @>>> X^{(2)} @<<< X_{\bar{\eta}}^{(2)} \\
@VVV @VVV @VVV \\
X_{s^{(1)}}^{(1)}\times_{X_s}X_{s^{(1)}}^{(1)} @>>> X^{(1)}\times_{X}X^{(1)} @<<<
X_{\bar{\eta}}^{(1)}\times_{X_{\bar{\eta}}}X_{\bar{\eta}}^{(1)} \\
\downdownarrows && \downdownarrows && \downdownarrows \\
X_{s^{(1)}}^{(1)} @>>> X^{(1)} @<<< X_{\bar{\eta}}^{(1)} \\
@VVV @VVV @VVV \\
X_s @>>> X @<<< X_{\bar{\eta}}
\end{CD}
\label{altdiagram}\end{equation}
where $X^{(1)} \to X$ is a semistable alteration and $X^{(2)} \to
X^{(1)}\times_{X}X^{(1)} $ is a proper morphism so that $X^{(2)} \to
(\overline{X^{(1)}\times_{X}X^{(1)}})_\eta$ is a semistable alteration onto the closure of the generic fibre. Here we denote by $S^{(j)}$ the base of the semistable family $X^{(j)}$ for $j=1,2$. From this diagram we then deduce a commutative diagram for cohomology
\begin{equation}
\begin{CD}
H^i_{rig}(X^{(2)}_{s^{(2)}}/k) @>c>> H^i_{HK}(X^{(2)}/S^{(2)})^{N=0} @>c_{st}>\simeq>
D_{crys}(H^i(X^{(2)}_{\bar{\eta}},\mathbb{Q}_p)) \\
\uparrow &&  && \uparrow \\
H^i_{rig}(X_{s^{(1)}}^{(1)}\times_{X_s}X_{s^{(1)}}^{(1)}/k) &&&&
D_{crys}(H^i(X_{\bar{\eta}}^{(1)}\times_{X_{\bar{\eta}}}X_{\bar{\eta}}^{(1)}
,\mathbb{Q}_p)) \\
\upuparrows &&&& \upuparrows \\
H^i_{rig}(X_{s^{(1)}}^{(1)}/k) @>c>> H^i_{HK}(X^{(1)}/S^{(1)})^{N=0} @>c_{st}>\simeq>
D_{crys}(H^i(X^{(1)}_{\bar{\eta}},\mathbb{Q}_p)) \\
\uparrow &&&& \uparrow \\
H^i_{rig}(X_s/k) &&&& D_{crys}(H^i(X_{\bar{\eta}},
\mathbb{Q}_p))
\end{CD}
\label{altinduced}\end{equation}
where the first arrow $c$ in each row is Chiarellotto's morphism (\ref{chiamap}) and $D_{crys}=D_{crys,K}$ is Fontaine's functor over $K$. Note here that we have
$$H^i_{rig}(X_{s^{(j)}}^{(j)}/k^{(j)})\cong H^i_{rig}(X_{s^{(j)}}^{(j)}/k)$$ for $j=1,2$, and the restriction of the base field from $k^{(j)}$ to $k$ just amounts to considering the $\phi$-module $H^i_{rig}(X_{s^{(j)}}^{(j)}/k^{(j)})$ over $K_0$ instead of $K_0^{(j)}$ (assuming $S^{(j)}$ is connected for the moment). Moreover we have an isomorphism of $G$-representations
\[ H^i(X^{(j)}_{\bar{\eta}},\mathbb{Q}_p)\cong \text{Ind}_{G^{(j)}}^GH^i(X^{(j)}_{\bar{\eta}^{(j)}},\mathbb{Q}_p))   \]
where $G^{(j)}=\Gal(\bar{K}/K^{(j)})\subseteq G$ and an isomorphism
\[ D_{crys,K^{(j)}}(V)\cong D_{crys,K}(\text{Ind}^G_{G^{(j)}}V)\]
for a $G^{(j)}$-representation $V$. If $S^{(j)}$ is not connected these considerations apply to each connected component. So source and target of the horizontal maps in (\ref{altinduced}) coincide with the groups discussed in the previous paragraph and the horizontal maps coincide with $\speci'$ in the semistable case.

\begin{prop} With the above notation, the equalizer of
$$H^i(X^{(1)}_{\bar{\eta}},\mathbb{Q}_p) \rightrightarrows H^i(X^{(2)}_{\bar{\eta}},\mathbb{Q}_p)
$$ is $H^i(X_{\bar{\eta}}, \mathbb{Q}_p)$ and the
equalizer of
$$D_{crys}(H^i(X^{(1)}_{\bar{\eta}},\mathbb{Q}_p))
\rightrightarrows
D_{crys}(H^i(X^{(2)}_{\bar{\eta}},\mathbb{Q}_p))
$$ is $D_{crys}(H^i(X_{\bar{\eta}},\mathbb{Q}_p))$.
\label{genericdescent}\end{prop}

\begin{proof} Clearly, the second statement follows from the first since $V\mapsto D_{crys}(V)$ is a left exact functor. The quickest way to see the first statement is to refer to the theory of cohomological descent \cite{hodgeiii} for varieties over fields of characteristic $0$. The diagram $X^{(2)}_{\bar{\eta}}\rightrightarrows X^{(1)}_{\bar{\eta}}\to X_{\bar{\eta}}$ can be extended to a proper smooth hyper-covering $X^{(\bullet+1)}_{\bar{\eta}}$ of $X_{\bar{\eta}}$ and the corresponding spectral sequence
\[ E_1^{r,s}=H^s(X^{(r+1)}_{\bar{\eta}},\bq_p) \Rightarrow H^{r+s}(X_{\bar{\eta}},\bq_p)\]
induces the weight filtration on $H^{r+s}(X_{\bar{\eta}},\bq_p)$ in the sense of \cite{hodgeiii}. But since $X_{\bar{\eta}}$ is assumed proper and smooth, $H^i(X_{\bar{\eta}},\bq_p)$ is pure of weight $i$ and the spectral sequence degenerates at $E_2$, inducing an isomorphism
\[ E_2^{0,i}=H^0(H^i(X^{(\bullet+1)}_{\bar{\eta}},\bq_p))\cong H^i(X_{\bar{\eta}},\bq_p)\]
which is just a reformulation of the statement that $H^i(X_{\bar{\eta}}, \mathbb{Q}_p)$ is the equalizer of
$$H^i(X^{(1)}_{\bar{\eta}},\mathbb{Q}_p) \rightrightarrows H^i(X^{(2)}_{\bar{\eta}},\mathbb{Q}_p).$$
\end{proof}

Now the left hand column in (\ref{altinduced}) induces maps
\begin{align}H^i_{rig}(X_s/k)\to &\text{eq}\left(H^i_{rig}(X_{s^{(1)}}^{(1)}/k) \rightrightarrows H^i_{rig}(X_{s^{(1)}}^{(1)}\times_{X_s}X_{s^{(1)}}^{(1)}/k)\right)\label{eqs}\\\to &\text{eq}\left(H^i_{rig}(X_{s^{(1)}}^{(1)}/k) \rightrightarrows H^i_{rig}(X_{s^{(2)}}^{(2)}/k)\right)\notag\end{align}
and by the commutativity of (\ref{altinduced}) this last equalizer maps to the equalizer  of
$$D_{crys}(H^i(X^{(1)}_{\bar{\eta}},\mathbb{Q}_p))\rightrightarrows
D_{crys}(H^i(X_{\bar{\eta}}^{(2)},\mathbb{Q}_p))$$ which coincides with $D_{crys}(H^i(X_{\bar{\eta}},
\mathbb{Q}_p))$ by the Lemma. Hence we obtain our map $\speci'$ by descent from the
map $\speci'$ constructed in the semistable case.
Note that $H^i_{rig}(X_{s^{(1)}}^{(1)}\times_{X_s}X_{s^{(1)}}^{(1)}/k)\to
H^i_{rig}(X^{(2)}_{s^{(2)}}/k) $ may not be injective and in fact all three groups in (\ref{eqs}) may be distinct.

\begin{lemma} The map $\speci':H^i_{rig}(X_s/k) \to D_{crys}(H^i(X_{\bar{\eta}}, \mathbb{Q}_p))$ is independent of the choice of alteration $X^{(1)}\to X$ and is functorial in $X$.
\label{functo}\end{lemma}

\begin{proof} In fact, given two semistable alterations $X^{(1)}\to X$ and $\tilde{X}^{(1)}\to X$ we can consider the disjoint union
$$\Tilde{\Tilde{X}}^{(1)}:= X^{(1)}\amalg \tilde{X}^{(1)}\to X\amalg X \xrightarrow{\nabla} X$$
which is also a semistable alteration. Denoting the specialization maps induced from these three alterations by $\speci'$, $\tilde{\speci}'$ and $\Tilde{\Tilde{\speci}}'$ respectively, we have a diagram
\[\begin{xy}
(35,45)*+{H^i_{rig}(\Tilde{\Tilde{X}}^{(1)}_{\Tilde{\Tilde{s}}^{(1)}}/k)}="a";
(120,45)*+{D_{crys}(H^i(\Tilde{\Tilde{X}}^{(1)}_{\bar{\eta}}, \mathbb{Q}_p))}="b";
(0,30)*+{H^i_{rig}(X^{(1)}_{s^{(1)}}/k)\oplus H^i_{rig}(\tilde{X}^{(1)}_{\tilde{s}^{(1)}}/k)}="c";
(80,30)*+{D_{crys}(H^i(X^{(1)}_{\bar{\eta}}, \mathbb{Q}_p))\oplus D_{crys}(H^i(\tilde{X}^{(1)}_{\bar{\eta}}, \mathbb{Q}_p))}="d";
(35,15)*+{H^i_{rig}(X_s/k)}="e";
(120,15)*+{D_{crys}(H^i(X_{\bar{\eta}}, \mathbb{Q}_p))}="f";
(0,0)*+{H^i_{rig}(X_s/k)\oplus H^i_{rig}(X_s/k)}="g";
(80,0)*+{D_{crys}(H^i(X_{\bar{\eta}}, \mathbb{Q}_p))\oplus
D_{crys}(H^i(X_{\bar{\eta}}, \mathbb{Q}_p))}="h";
{\ar^(.45){\Tilde{\Tilde{c}}} "a";"b"};{\ar@{=} "a";"c"};{\ar^(.44){c\oplus\tilde{c}} "c";"d"};{\ar@{=} "b";"d"};{\ar^(.45){\Tilde{\Tilde{\speci}}'} |>>>>>>>>>>>>>>>>>>>>{\hole}"e";"f"};{\ar_{\Delta} "f";"h"};{\ar_{\Delta} "e";"g"};{\ar^(.447){\speci' \oplus \tilde{\speci}'}  "g";"h"};{\ar |>>>>>>>>{\object=(1,6){}} "e";"a"};{\ar "g";"c"};{\ar "f";"b"};{\ar^(.4){\alpha} "h";"d"};
\end{xy}\]
where $\Delta$ is the diagonal map $x \to (x,x)$ and $c,\tilde{c}$ and $\Tilde{\Tilde{c}}=c\oplus\tilde{c}$ are the lower horizontal maps in (\ref{altinduced}) for the respective alterations. By construction of the specialization map, all squares except the bottom one commute. An easy diagram chase using the injectivity of  $\alpha$ then implies that the bottom square commutes as well, hence $\speci'=\tilde{\speci}'=\Tilde{\Tilde{\speci}}'$.

If $Y\to X$ is an $S$-morphism between generically smooth, proper flat $S$-schemes, $X^{(1)}\to X$ a semistable alteration we can choose a semistable alteration $Y^{(1)}\to (\overline{Y\times_XX^{(1)}})_\eta$ onto the closure of the generic fibre. Then $Y^{(1)}\to Y$ will also be a semistable alteration and we obtain a commutative diagram
\[\begin{CD} Y^{(1)} @>>> X^{(1)}\\ @VVV @VVV\\ Y @>>> X\end{CD}\]
inducing a cubical diagram of maps similar to the one above, the bottom of which is the diagram
\[\begin{CD} H^i_{rig}(X_s/k) @>\speci'>> D_{crys}(H^i(X_{\bar{\eta}}, \mathbb{Q}_p)) \\
@VVV @VVV\\
H^i_{rig}(Y_s/k) @>\speci'>> D_{crys}(H^i(Y_{\bar{\eta}}, \mathbb{Q}_p)).
\end{CD}\]
All other sides commute by known functorialities and the injectivity of $$D_{crys}(H^i(Y_{\bar{\eta}}, \mathbb{Q}_p))\to D_{crys}(H^i(Y^{(1)}_{\bar{\eta}}, \mathbb{Q}_p))$$ then implies that the bottom diagram commutes as well.
\end{proof}

By these we contruct $\speci':H^i_{rig}(X_s/k) \to D_{crys}(H^i(X_{\bar{\eta}}, \mathbb{Q}_p))$ in general and get Proposition \ref{speciconstruct}.

\begin{rem}
	
Note that we have trace maps on the generic fiber induced by the following lemma

\begin{lemma} Let  $\xi: V'\to V$ be a surjective morphism between smooth proper varieties of the same dimension $n$ over a field $K$ of characteristic $0$. Then there is a trace map  $\trace=\trace_\xi: R\xi_* \mathbb{Q}_p \to \mathbb{Q}_p$ so that the composite
	$\bq_p \to R\xi_* \mathbb{Q}_p \xrightarrow{\trace} \mathbb{Q}_p$ is multiplication by the generic degree $d_\xi$ (a locally constant function on $V$).
	\label{tracelemma}\end{lemma}
\begin{proof} This follows from the six functor formalism on the derived category of $p$-adic \'etale sheaves.  If $\eta: V \to \Spec(K)$, $\eta' = \eta \circ \xi: V' \to \Spec(K)$ denote the structure maps, we have $$\bq_p = (\eta')^!\mathbb{Q}_p(-n)[-2n] =
	\xi^!\eta^!\mathbb{Q}_p(-n)[-2n] = \xi^!\mathbb{Q}_p$$ since $\eta$ and $\eta'$ are smooth, and we have $R\xi_! = R\pi_*$ as $\xi$ is proper. So the trace map is simply
	$R\xi_*\bq_p\cong R\xi_!\xi^!\mathbb{Q}_p \to \mathbb{Q}_p $ where the last arrow is the adjunction. If $\xi$ is finite \'etale, the composite $\bq_p \to R\xi_* \mathbb{Q}_p \xrightarrow{\trace} \mathbb{Q}_p$ is multiplication by $d_\xi$ \cite{miletale}[V.1.12] and in general this is true on the dense open $U\subseteq V$ where $\xi$ is finite \'etale. But maps between (locally) constant $p$-adic sheaves agree if they agree on a dense open subset.
\end{proof}

As we mentioned, Conjecture \ref{slope01} follows from the compatibility diagram of trace maps (\ref{trace compatibility slope01}). A possible approach is to show that these trace morphisms satisfy some required homotopy identities. In fact, we mention here that Proposition \ref{genericdescent} can also be proved via homotopy argument of trace maps on the generic fiber, and we expect analogue results of Proposition \ref{genericdescent} for Witt vectors cohomology in the special fiber, i.e. the equalizer of
$H^i(X^{(1)}, W_n\co_{X^{(1)}})  \rightrightarrows H^i(X^{(2)}, W_n\co_{X^{(2)}}) )
$ is $H^i(X, W_n\co_{X}) $. However, the proof does not follow directly from the one on generic fiber, the difficulty is the following: Given
a semistable alteration $Y\to X$, the scheme $Y\times_XY$ might have
irreducible components lying entirely in the special fibre. Hence
one cannot choose a second semistable alteration $Y'\to Y\times_XY$
that is surjective. Another difficulty is that $Y\times_X Y\to Y$ is
not necessarily a local complete intersection morphism, and hence
there is a priori no trace map.

 Moreover, if we consider the cospecialization map
 \[ H^i_{rig}(X_s/k) \to H^i_{dR}(X_\eta/K)\cong D_{dR}(H^i(X_{\bar{\eta}}, \mathbb{Q}_p))\] defined in
 \cite{bcf04}, the compatibility (\ref{trace compatibility slope01}) is then the corollary of a
 more general compatibility of trace maps:
 \[
 \begin{CD}
 H^i_{rig}(X^{(1)}_s/k) @>>> D_{dR}(H^i(X^{(1)}_{\bar{\eta}}, \mathbb{Q}_p)) \\
 @V\tau_s VV   @V\tau_{\eta}VV \\
 H^i_{rig}(X_s/k) @>>> D_{dR}(H^i(X_{\bar{\eta}}, \mathbb{Q}_p))
 \end{CD}
 \]
 provided we could define the trace map $\tau_s:
 H^i_{rig}(X^{(1)}_s/k) \to H^i_{rig}(X_s/k)$ for rigid cohomology.
 The compatibility of the trace map would then implies the compatibility (\ref{trace compatibility slope01}) and deduce Conjecture \ref{slope01}, also some results in
 \cite{beresrue10}.

 In some special cases, e.g. $X^{(1)} \to X $ is finite and flat,
 Grosse-Kl\"onne in \cite{grosse02} defined a trace map for rigid cohomology(which satisfies the compability). The construction of a trace map for rigid cohomology
 in a general case is still unknown.

\end{rem}


\bigskip

Fakult\"at f\"ur Mathematik, Universit\"at Regensburg, 93040 Regensburg, Germany

Email address:  Yitao.Wu@mathematik.uni-regensburg.de

\end{document}